\documentclass[12pt, reqno, twoside, letterpaper]{amsart}

\usepackage{paperstyle}
\usepackage{graphicx}

\usepackage{mathtools}
\usepackage{todonotes}
\usepackage[norefs,nocites]{refcheck}

\newcommand{\be}{\begin{equation}}
\newcommand{\ee}{\end{equation}}
\newcommand{\dalign}[1]{\[\begin{aligned} #1 \end{aligned}\]}

\newcommand{\euB}{\EuScript{B}}

\newcommand{\euZ}{\EuScript{Z}}
\newcommand{\er}{\mathrm{e}}

\title[GRH from zeros of $\zeta(s)$]
{The Generalized Riemann Hypothesis\\ from zeros of
the zeta function}
      
\author[W.~Banks]{William Banks}

\address{Department of Mathematics, 
         University of Missouri, 
         Columbia MO 65211, USA.}

\email{bankswd@missouri.edu}
        
\date{\today}

\begin{document}

\begin{abstract}
We show that the Generalized Riemann Hypothesis for
Dirichlet $L$-functions is a consequence of certain
conjectural properties of the zeros of the Riemann
zeta function. Conversely, we prove that the zeros of $\zeta(s)$
satisfy those properties under GRH.
\end{abstract}




\thanks{Data Availability Statement: Data sharing not applicable to this article as no datasets
were generated or analysed during the current study.}

\thanks{Potential Conflicts of Interest: NONE}

\thanks{Research Involving Human Participants and/or Animals: NONE}

\maketitle



\centerline{\it Dedicated to John Friedlander and Henryk Iwaniec}

\section{Introduction}
\label{sec:intro}

In a recent paper \cite{Banks}, the author has shown that
all Dirichlet $L$-functions $L(s,\chi)$ are nonvanishing
in the half-plane $\sigma>\tfrac{9}{10}$ if and only if
the zeros of the Riemann zeta function $\zeta(s)$ satisfy
certain distributional properties, including the classical
Riemann Hypothesis (RH). In the present paper, we extend that
result to show that the full Generalized Riemann
Hypothesis (GRH) holds if and only if the zeros of $\zeta(s)$
satisfy similar criteria.

To formulate our main result, we introduce some notation.
Throughout the paper, let $C_c^\infty(\R^+)$ be the space of smooth
functions $\euB:\R^+\to\C$ with compact support in $\R^+$.
Let $\mu$ and $\phi$ be the M\"obius and Euler functions, respectively. 
Also, let us recall the functional equation
$\zeta(s)=\cX(s)\zeta(1-s)$, where
$$
\cX(1-s)=\cX(s)^{-1}
=\pi^{\frac12-s}\frac{\Gamma(\frac12s)}{\Gamma(\frac12-\frac12s)}
=2(2\pi)^{-s}\Gamma(s)\cos\tfrac12\pi s.
$$

\begin{theorem}\label{thm:RHvsGRH}
The following statements are equivalent:
\begin{itemize}
\item[$(i)$] The Generalized Riemann Hypothesis holds for all Dirichlet $L$-functions;
\item[$(ii)$] The Riemann Hypothesis is true, and
for any rational number $\xi\defeq m/q$ with $0<m<q$
and $(m,q)=1$, any $\euB\in C_c^\infty(\R^+)$, and $\eps>0$, we have
\be\label{eq:superbound}
\sum_{\rho=\frac12+i\gamma}\xi^{-\rho}\,\cX(1-\rho)
\euB\Big(\frac{\gamma}{2\pi X}\Big)
+\frac{\mu(q)}{\phi(q)}C_\euB X
\,\mathop{\ll}\limits_{\xi,\euB,\eps}\,X^{1/2+\eps},
\ee
where the sum on the left runs over all complex zeros
$\rho=\frac12+i\gamma$ of $\zeta(s)$ $($counted with multiplicity$)$,
$C_\euB\defeq\int_0^\infty\euB(u)\,du$,
and the implied constant depends only on
$\xi$, $\euB$, and $\eps$.
\end{itemize}
\end{theorem}

We emphasize that statement $(ii)$ involves
only properties of the Riemann zeta function. In particular,
\emph{no conditions are placed on the zeros
of \text{non-principal} Dirichlet \text{$L$-functions}}.

Theorem~\ref{thm:RHvsGRH} is closely related to an
earlier result of Sprind\v zuk~\cite{Sprind1,Sprind2},
which he obtained by developing ideas of Linnik~\cite{Linnik}.
Under RH, Sprind\v zuk showed that GRH is equivalent to the
assertion that, for any rational number $\xi\defeq m/q$ with
$0<|m|\le q/2$ and $(m,q)=1$, and any $\eps>0$, one has
\be\label{eq:superbound2}
\sum_{\rho=\frac12+i\gamma}|\gamma|^{i\gamma}
\er^{-i\gamma-\pi|\gamma|/2}(x+2\pi i\xi)^{-\rho}
+\frac{\mu(q)}{\phi(q)}\frac{1}{x\sqrt{2\pi}}
\,\mathop{\ll}\limits_{\xi,\eps}\,x^{-1/2-\eps}
\ee
as $x$ tends to $0^+$. The Linnik–Sprind\v zuk theorem says
(essentially) that the GRH is equivalent to RH together
with a suitable property about the vertical
distribution of the zeros of $\zeta(s)$, hence the zeros of
$\zeta(s)$ contain information about the zeros of 
$L(s,\chi)$, and vice versa. Using various methods,
similar results in this direction have been attained by
Fujii~\cite{Fujii1,Fujii2,Fujii3},
Suzuki~\cite{Suzuki}, and Kaczorowski and Perelli~\cite{KacPer}.
Theorem~\ref{thm:RHvsGRH} provides another instance of this idea.

We remark that, in place of \eqref{eq:superbound},
Banks \cite[Thm.~1.1]{Banks} gives the weaker
bound\footnote{Here, we correct a typo in the statement of
\cite[Thm.~1.1]{Banks},
where $\xi^{-i\gamma}$ appears instead of $\xi^{-\rho}$.}
$$
\sum_{\rho=\frac12+i\gamma}\xi^{-\rho}\,\euZ(\rho)
\euB\Big(\frac{\gamma}{2\pi X}\Big)
+\frac{\mu(q)}{\phi(q)}\sum_n\Lambda(n)\euB(n/X)
\,\mathop{\ll}\limits_{\xi,\euB,\eps}\,X^{9/10+\eps}.
$$
One can easily verify that each term $\euZ(\rho)$
(as defined in \cite{Banks}) is equal to $\cX(1-\rho)$.
Moreover, we show that
$\sum_n\Lambda(n)\euB(n/X)$ is equal to $C_\euB X$ up to
an acceptable error; see Lemma~\ref{lem:blue}.
Thus, statement $(ii)$ in Theorem~\ref{thm:RHvsGRH} 
is consistent the results of \cite{Banks}.

\section{Preliminaries}

\subsection{Notation}

We write $\e(u)\defeq \er^{2\pi iu}$
for all $u\in\R$. As in  Theorem~\ref{thm:RHvsGRH}, $C_\euB$
is used to denote the constant $\int_0^\infty\euB(u)\,du$.
 
Throughout the paper, any implied constants in the
symbols $\ll$, $O$, etc., may depend on various parameters
(e.g., $\xi$, $\euB$, and $\eps$), and this is
indicated by the notation (see, e.g., \eqref{eq:superbound});
such constants are independent of all other parameters.

\subsection{The function $\cX(s)$}

For fixed $\delta>0$, Stirling's formula for the gamma function,
namely
$$
\Gamma(z)=\sqrt{2\pi}\,z^{z-1/2}\er^{-z}
\big\{1+O_\delta(|z|^{-1})\big\},
$$
holds uniformly for all complex $z$ satisfying
$|z|\ge\delta$ and $|{\rm arg}~z|<\pi-\delta$;
see, for example, Montgomery and Vaughan~\cite[Thm.~C.1]{MontVau}.
From this, it is not difficult to derive the estimates
\be\label{eq:X1-s}
\cX(1-s)=\er^{-i\pi/4}\Big(\frac{t}{2\pi}\Big)^{\sigma-1/2}
\exp\Big(it\log\Big(\frac{t}{2\pi\er}\Big)\Big)
\big\{1+O_\cI(t^{-1})\big\}
\ee
and
\be\label{eq:cXs}
\cX(\overline{s})=\er^{-i\pi/4}\Big(\frac{t}{2\pi}\Big)^{1/2-\sigma}
\exp\Big(it\log\Big(\frac{t}{2\pi\er}\Big)\Big)
\big\{1+O_\cI(t^{-1})\big\}
\ee
uniformly for all $\sigma$ in a given bounded interval $\cI$
and all $t\ge 1$. The following lemma
is due to Gonek~\cite[Lem.~2]{Gonek}; the proof uses
\eqref{eq:X1-s} and the stationary phase method.
The formulation given here is a special case 
of Conrey, Ghosh and Gonek~\cite[Lem.~1]{ConGhoGon}.

\begin{lemma}\label{lem:ConGhoGon}
For any $r>0$ and $\tfrac1{100}\le c\le 2$, we have
$$
\frac{1}{2\pi i}\int_{c+i}^{c+iT}r^{-s}\cX(1-s)\,ds=\begin{cases}
\e(-r)+E(r,c)r^{-c}&\quad\hbox{if $r\le T/2\pi$},\\
E(r,c)r^{-c}&\quad\hbox{otherwise},\\
\end{cases}
$$
where
$$
E(r,c)\ll T^{c-1/2}+\frac{T^{c+1/2}}{|T-2\pi r|+T^{1/2}}.
$$
\end{lemma}

\subsection{Reformulation of GRH}

For a given Dirichlet character $\chi\bmod q$,
its Dirichlet \text{$L$-function} is defined 
for $\sigma>1$ by
$$
L(s,\chi)\defeq\sum_{n\ge 1} \chi(n)n^{-s}
=\prod_{p\text{~prime}}(1-\chi(p)p^{-s})^{-1}.
$$
Consider the following hypothesis about the zeros of $L(s,\chi)$:

\bigskip\noindent{\sc Hypothesis ${\rm GRH}[\chi]$}:
{\it If $L(\beta+i\gamma,\chi)=0$ and $\beta>0$, then $\beta=\tfrac12$.}

\bigskip\noindent The Generalized Riemann Hypothesis is the
assertion that ${\rm GRH}[\chi]$ holds for all
characters $\chi$. Note that for a principal
character $\chi_0\bmod q$, we have
$$
L(s,\chi_0)=\zeta(s)\prod_{p\,\mid\,q}(1-p^{-s}),
$$
hence ${\rm GRH}[\chi_0]$ is equivalent to RH. Similarly, if
$\chi$ is induced from a primitive character~$\chi_*$
of conductor $q_*>1$, then $L(s,\chi)$ and $L(s,\chi_*)$ have the same zeros
in the critical strip since
$$
L(s,\chi)=L(s,\chi_*)
\sprod{p\,\mid\,q\\p\,\nmid\,q_*}(1-\chi_*(p)p^{-s}),
$$
and therefore ${\rm GRH}[\chi]$ is equivalent to ${\rm GRH}[\chi_*]$.
From these remarks, and taking into account 
the special case $\beta_0\defeq\tfrac12$ of \cite[Lem.~5.1]{Banks},
the next result is immediate.

\begin{lemma}\label{lem:ultraclean}
The following statements are equivalent:
\begin{itemize}
\item[$(i)$] The Generalized Riemann Hypothesis holds for all Dirichlet $L$-functions;
\item[$(ii)$] The Riemann Hypothesis is true, and
for every primitive character $\chi$ of modulus $q>1$,
any $\euB\in C_c^\infty(\R^+)$, and $\eps>0$, we have
$$
\sum_n\Lambda(n)\chi(n)\euB(n/X)
\,\mathop{\ll}\limits_{q,\euB,\eps}\, X^{1/2+\eps}.
$$
\end{itemize}
\end{lemma}

\subsection{Initial estimates}

\begin{lemma}\label{lem:blue}
Assume \text{\rm RH}.
For every $\euB\in C_c^\infty(\R^+)$, we have
$$
\sum_n\Lambda(n)\euB(n/X)=C_\euB X+O_\euB(X^{1/2}\log^2X).
$$
\end{lemma}

\begin{proof}
Under RH, the Chebyshev function
$\psi(x)\defeq\sum_{n\le x}\Lambda(n)$ is known to satisfy
\be\label{eq:Chebest}
\psi(x)=x+E(x)\qquad\text{with}\quad
E(x)\ll x^{1/2}\log^2x;
\ee
see \cite[Thm.~13.1]{MontVau}.
Using Riemann-Stieltjes integration, we have
\dalign{
\sum_n\Lambda(n)\euB(n/X)&=\int_0^\infty\euB(u/X)\,d\psi(u)
=\int_0^\infty\euB(u)\,d\psi(Xu)\\
&=\int_0^\infty\euB(u)\,d(Xu)+\int_0^\infty\euB(u)\,dE(Xu)\\
&=C_\euB X-\int_0^\infty\euB'(u)E(Xu)\,du.
}
The bound in \eqref{eq:Chebest} implies that the last integral is
$O_\euB(X^{1/2}\log^2X)$, and the lemma is proved.
\end{proof}

\begin{lemma}\label{lem:green}
Assume GRH. For any rational number $\xi\defeq m/q$ with $0<m<q$
and $(m,q)=1$, $\euB\in C_c^\infty(\R^+)$, and $\eps>0$,
we have
$$
\sum_n\Lambda(n)\e(-n\xi)\euB(n/X)
=\frac{\mu(q)}{\phi(q)}C_\euB X
+O_{\xi,\euB,\eps}(X^{1/2+\eps})
$$
\end{lemma}

\begin{proof}
In view of Lemma~\ref{lem:blue},
this is a consequence of \cite[Lem.~5.2]{Banks}.
\end{proof}

\section{Twisting the von Mangoldt function}
\label{sec:vonMangoldt-twist}

\begin{theorem}\label{thm:vonMangoldt-twist}
For any $\xi\in\R^+$ and $T\ge 100$, we have
$$
\ssum{\rho=\beta+i\gamma\\0<\gamma\le T}\xi^{-\rho}\cX(1-\rho)
+\sum_{n\le T/2\pi\xi}\Lambda(n)\e(-n\xi)
\,\mathop{\ll}\limits_\xi\,T^{1/2}\log^2T.
$$
where the first sum runs over complex zeros
$\rho=\beta+i\gamma$ of $\zeta(s)$ $($counted with multiplicity$)$,
and the implied constant depends only on $\xi$.
\end{theorem}

\begin{proof}
For $u>1$, let
$$
\Sigma_1(u)\defeq\ssum{\rho=\beta+i\gamma\\0<\gamma\le u}
\xi^{-\rho}\cX(1-\rho),\qquad
\Sigma_2(u)\defeq\ssum{1<n\le u}\Lambda(n)\e(-n\xi).
$$
Then our goal is to show that
\be\label{eq:Sigsum}
\Sigma_1(T)+\Sigma_2(T/2\pi\xi)
\,\mathop{\ll}\limits_\xi\,T^{1/2}\log^2T.
\ee

According to \cite[Lemma~12.2]{MontVau},
for any real number $T_*\ge 2$, there exists $T\in[T_*,T_*+1]$ such that
\be\label{eq:LD-horiz}
\frac{\zeta'}{\zeta}(\sigma+iT)\ll\log^2T\qquad(-1\le\sigma\le 2).
\ee
In proving \eqref{eq:Sigsum}, we can assume 
that $T$ has this property. Indeed, let $T_*\ge 100$ be
arbitrary, and suppose $T\in[T_*,T_*+1]$ satisfies
\eqref{eq:Sigsum}.
By \eqref{eq:X1-s}, we have $\cX(1-\rho)\ll |\gamma|^{1/2}$
for all complex zeros $\rho=\beta+i\gamma$ of $\zeta(s)$,
and therefore
$$
\big|\Sigma_1(T_*)-\Sigma_1(T)\big|\le
\ssum{\rho=\beta+i\gamma\\T_*<\gamma\le T_*+1}
\big|\xi^{-\rho}\cX(1-\rho)\big|\,\mathop{\ll}\limits_\xi\,T_*^{1/2}
\log T_*
$$
since the number of zeros
$\rho=\beta+i\gamma$ of $\zeta(s)$ with $T_*<\gamma\le T_*+1$
does not exceed $O(\log T_*)$. We also have
$$
\big|\Sigma_2(T_*/2\pi\xi)-\Sigma_2(T/2\pi\xi)\big|\le
\ssum{T_*/2\pi\xi<n\le (T_*+1)/2\pi\xi}\big|\Lambda(n)\e(-n\xi)\big|\,\mathop{\ll}\limits_\xi\,\log T_*.
$$
Combining the preceding two bounds and \eqref{eq:Sigsum},
it follows that
$$
\Sigma_1(T_*)+\Sigma_2(T_*/2\pi\xi)
\,\mathop{\ll}\limits_\xi\,T_*^{1/2}\log^2T_*.
$$

From now on, we assume that $T$ satisfies \eqref{eq:LD-horiz}.
In particular, $T$ is not the ordinate of a zero of $\zeta(s)$.
Put $c\defeq 1+\frac{1}{\log T}$,
and let $\cC$ be the rectangle in $\C$:
$$
c+i
~~\longrightarrow~~c+iT
~~\longrightarrow~~-1+iT
~~\longrightarrow~~-1+i
~~\longrightarrow~~c+i.
$$
By Cauchy's theorem, we have
\dalign{
\Sigma_1(T)
&=\frac{1}{2\pi i}\bigg(\int_{c+i}^{c+iT}
+\int_{c+iT}^{-1+iT}
+\int_{-1+iT}^{-1+i}
+\int_{-1+i}^{c+i}\bigg)
\frac{\zeta'}{\zeta}(s)\,\xi^{-s}\cX(1-s)\,ds\\
&=I_1+I_2+I_3+I_4\quad\text{(say)}.
}
We estimate the four integrals separately.

Since
$$
\frac{\zeta'}{\zeta}(s)=-\sum_n\frac{\Lambda(n)}{n^s}
\qquad(\sigma>1),
$$
it follows that
$$
I_1=-\sum_{n}\Lambda(n)\cdot\frac{1}{2\pi i}
\int_{c+i}^{c+iT}(n\xi)^{-s}\cX(1-s)\,ds.
$$
Applying Lemma~\ref{lem:ConGhoGon}, we get that
$$
I_1=-\sum_{n\xi\le T/2\pi}\Lambda(n)\e(-n\xi)
+\cc{O\bigg(}\xi^{-c}\sum_n\frac{\Lambda(n)}{n^c}T^{c-1/2}
+\xi^{-c}\sum_n\frac{\Lambda(n)}{n^c}
\frac{T^{c+1/2}}{|T-2\pi n\xi|+T^{1/2}}\cc{\bigg)}.
$$
The first sum is $\Sigma_2(T/2\pi\xi)$; hence, defining
$T_\circ\defeq T/2\pi\xi$ and noting that
\be\label{eq:brilliant}
\sum_n\frac{\Lambda(n)}{n^c}
=-\frac{\zeta'}{\zeta}(c)\ll \frac{1}{c-1}=\log T,
\ee
we see that
\be\label{eq:fantasy1}
I_1+\Sigma_2(T/2\pi\xi)\,\mathop{\ll}\limits_\xi\,
T^{1/2}\log T
+T^{3/2}\sum_{n\ge 2}\frac{\Lambda(n)}{n^c}
\frac{1}{|n-T_\circ|+T^{1/2}}.
\ee
For the last sum, we split the integers $n\ge 2$ into
three disjoint sets:
\dalign{
S_1&\defeq\{n\ge 2:|n-T_\circ|>\tfrac12T_\circ\},\\
S_2&\defeq\{n\ge 2:|n-T_\circ|\le T_\circ^{1/2}\},\\
S_3&\defeq\{n\ge 2:T_\circ^{1/2}<|n-T_\circ|\le \tfrac12T_\circ\}.
}
Since $|n-T_\circ|+T^{1/2}\gg_\xi T$
for each $n\in S_1$, we have by \eqref{eq:brilliant}:
\be\label{eq:fantasy2}
\sum_{n\in S_1}\frac{\Lambda(n)}{n^c}\frac{1}{|n-T_\circ|+T^{1/2}}
\,\mathop{\ll}\limits_\xi\,\frac{\log T}{T},
\ee
which is acceptable. For each $n\in S_2$, we have
$$
n\asymp T_\circ\,\mathop{\asymp}\limits_\xi\, T,\qquad
\frac{\Lambda(n)}{n^c}
\,\mathop{\ll}\limits_\xi\,\frac{\log T}{T^c}
\asymp\frac{\log T}{T},\qquad
\frac{1}{|n-T_\circ|+T^{1/2}}
\,\mathop{\ll}\limits_\xi\,T^{-1/2}.
$$
As $|S_2|\ll T^{1/2}$, it follows that
\be\label{eq:fantasy3}
\sum_{n\in S_2}\frac{\Lambda(n)}{n^c}\frac{1}{|n-T_\circ|+T^{1/2}}
\,\mathop{\ll}\limits_\xi\,\frac{\log T}{T}.
\ee
Finally, for each $n\in S_3$, we again have
$n\asymp T_\circ\,\mathop{\asymp}\limits_\xi\,T$, and so
$$
\sum_{n\in S_3}\frac{\Lambda(n)}{n^c}\frac{1}{|n-T_\circ|+T^{1/2}}
\,\mathop{\ll}\limits_\xi\,\frac{\log T}{T}
\sum_{\cc{n\in S_3}}\frac{1}{|n-T_\circ|+T^{1/2}}.
$$
The last sum is bounded by
$$
\ll\sum_{T_\circ^{1/2}\le k\le \frac12T_\circ}
\ssum{n\ge 1\\k<|n-T_\circ|\le k+1}\frac{1}{|n-T_\circ|+T^{1/2}}
\ll\sum_{T_\circ^{1/2}\le k\le \frac12T_\circ}\frac{1}{k}
\,\mathop{\ll}\limits_\xi\,\log T,
$$
hence it follows that
\be\label{eq:fantasy4}
\sum_{n\in S_3}\frac{\Lambda(n)}{n^c}\frac{1}{|n-T_\circ|+T^{1/2}}
\,\mathop{\ll}\limits_\xi\,\frac{\log^2T}{T}.
\ee
Combining the bounds \eqref{eq:fantasy1}$-$\eqref{eq:fantasy4},
we have therefore shown that
\be\label{eq:I1est}
I_1=-\Sigma_2(T/2\pi\xi)+O_\xi(T^{1/2}\log^2T).
\ee

Next, observe that \eqref{eq:X1-s} yields the uniform bound
$$
\cX(1-(\sigma+iT))\ll T^{1/2}\qquad(-1\le\sigma\le c).
$$
Recalling \eqref{eq:LD-horiz}, it follows that
\be\label{eq:I2est}
I_2=-\frac{1}{2\pi i}\int_{-1+iT}^{c+iT}
\frac{\zeta'}{\zeta}(s)\,\xi^{-s}\cX(1-s)\,ds
\,\mathop{\ll}\limits_\xi\,T^{1/2}\log^2T.
\ee
Similarly, using \eqref{eq:cXs} we derive that
$$
\cX(2-it)\ll t^{-3/2}\qquad(1\le t\le T).
$$
Moreover, according to \cite[Lemma~12.4]{MontVau},
the following bound holds:
$$
\frac{\zeta'}{\zeta}(-1+it)\ll\log 2t\qquad(1\le t\le T).
$$
Therefore,
\be\label{eq:I3est}
I_3=-\frac{1}{2\pi i}\int_{-1+i}^{-1+iT}
\frac{\zeta'}{\zeta}(s)\,\xi^{-s}\cX(1-s)\,ds
\ll\xi^{-1}\int_1^T t^{-3/2}\log 2t\,dt
\,\mathop{\ll}\limits_\xi\,1.
\ee
Finally, it is clear that
\be\label{eq:I4est}
I_4=\frac{1}{2\pi i}\int_{1-c+i}^{c+i}
\frac{\zeta'}{\zeta}(s)\,\xi^{-s}\cX(1-s)\,ds
\,\mathop{\ll}\limits_\xi\,1.
\ee

Combining \eqref{eq:I1est}$-$\eqref{eq:I4est}, we obtain
\eqref{eq:Sigsum}, and the proof is complete.
\end{proof}

\begin{corollary}\label{cor:vonMangoldt-twist}
For any $\xi\in\R^+$, $\euB\in C_c^\infty(\R^+)$, and $X>100\xi^{-1}$,
we have
\be\label{eq:yelp}
\ssum{\rho=\beta+i\gamma}\xi^{-\rho}
\cX(1-\rho)\euB\Big(\frac{\gamma}{2\pi\xi X}\Big)
+\sum_n\Lambda(n)\e(-n\xi)\euB(n/X)
\,\mathop{\ll}\limits_{\xi,\euB}\,X^{1/2}\log^2X.
\ee
\end{corollary}

\begin{proof}
Denote
$$
F_1(u)\defeq\ssum{\rho=\beta+i\gamma\\0<\gamma\le u}
\xi^{-\rho}\cX(1-\rho),\qquad
F_2(u)\defeq\ssum{n\le u}\Lambda(n)\e(-n\xi).
$$
Applying Theorem~\ref{thm:vonMangoldt-twist} with $T\defeq 2\pi\xi Xu$,
we have the uniform bound
\be\label{eq:lawnmower}
F_1(2\pi\xi Xu)+F_2(Xu)\,\mathop{\ll}\limits_\xi\,(Xu)^{1/2}\log^2(Xu)
\qquad(u\ge\tfrac12,~X>100\xi^{-1}).
\ee

Using Riemann-Stieltjes integration, we have
\dalign{
\ssum{\rho=\beta+i\gamma}\xi^{-\rho}
\cX(1-\rho)\euB\Big(\frac{\gamma}{2\pi\xi X}\Big)
&=\int_0^\infty\euB\Big(\frac{u}{2\pi\xi X}\Big)\,dF_1(u)
=\int_0^\infty\euB(u)\,dF_1(2\pi\xi Xu)\\
&=-\int_0^\infty\euB'(u)F_1(2\pi\xi Xu)\,du,
}
and
\dalign{
\sum_n\Lambda(n)\e(-n\xi)\euB(n/X)
&=\int_0^\infty\euB(u/X)\,dF_2(u)
=\int_0^\infty\euB(u)\,dF_2(Xu)\\
&=-\int_0^\infty\euB'(u)F_2(Xu)\,du,
}
and so the left side of \eqref{eq:yelp} is equal to
$$
-\int_0^\infty\euB'(u)\Big\{F_1(2\pi\xi Xu)+F_2(Xu)\Big\}\,du.
$$
Using \eqref{eq:lawnmower}, the result follows immediately.
\end{proof}

\section{Proof of Theorem~\ref{thm:RHvsGRH}}

First, we show $(i)\Longrightarrow(ii)$. Assume GRH holds for all
Dirichlet $L$-functions. Since $\zeta(s)$ is one such $L$-function,
RH is true.

Next, suppose we are given a rational number
$\xi\defeq m/q$ with $0<m<q$ and $(m,q)=1$, a function
$\euB\in C_c^\infty(\R^+)$, and $\eps>0$. By Lemma~\ref{lem:green},
\be\label{eq:drum}
\sum_n\Lambda(n)\e(-n\xi)\euB(n/X)
=\frac{\mu(q)}{\phi(q)}C_\euB X
+O_{\xi,\euB,\eps}(X^{1/2+\eps}),
\ee
and Corollary~\ref{cor:vonMangoldt-twist} shows that
\be\label{eq:stick}
\ssum{\rho=\frac12+i\gamma}\xi^{-\rho}
\cX(1-\rho)\euB\Big(\frac{\gamma}{2\pi\xi X}\Big)
+\sum_n\Lambda(n)\e(-n\xi)\euB(n/X)
\,\mathop{\ll}\limits_{\xi,\euB}\,X^{1/2}\log^2X.
\ee
Combining these results, we deduce the bound \eqref{eq:superbound},
and this completes our proof that $(i)\Longrightarrow(ii)$.

To prove the reverse implication $(ii)\Longrightarrow(i)$, assume RH
and the bound \eqref{eq:superbound} for all triples
$(\xi,\euB,\eps)$ as above. Applying
Corollary~\ref{cor:vonMangoldt-twist} again, which yields the
bound \eqref{eq:stick}, and using \eqref{eq:superbound},
we deduce that \eqref{eq:drum} again holds.

According to Lemma~\ref{lem:ultraclean}, to establish GRH, it suffices
to show that for any primitive character $\chi$ modulo $q>1$,
any $\euB\in C_c^\infty(\R^+)$, and $\eps>0$, we have
\be\label{eq:goalbd}
\sum_n\Lambda(n)\chi(n)\euB(n/X)
\,\mathop{\ll}\limits_{q,\euB,\eps}\, X^{1/2+\eps}.
\ee
For this, we invoke the following identity
(see, e.g., Bump~\cite[Chapter 1]{Bump}):
\be\label{eq:chi-basic}
\chi(n)=\frac{\chi(-1)\tau(\chi)}{q}
\sum_{m\bmod q}\overline\chi(m)\e(-nm/q)
\qquad(n\in\Z),
\ee
where the sum runs over any complete set of reduced residue
classes $m\bmod q$, and
$\tau(\chi)$ is the Gauss sum given by
$$
\tau(\chi)\defeq\sum_{n\bmod q}\chi(n)\e(n/q).
$$
Using \eqref{eq:drum} and \eqref{eq:chi-basic}, we get that
\dalign{
\sum_n\Lambda(n)\chi(n)\euB(n/X)
&=\frac{\chi(-1)\tau(\chi)}{q}\ssum{0<m<q\\(m,q)=1}\overline\chi(m)
\sum_n\Lambda(n)\e(-nm/q)\euB(n/X)\\
&=\frac{\chi(-1)\tau(\chi)}{q}\ssum{0<m<q\\(m,q)=1}\overline\chi(m)
\bigg(\frac{\mu(q)}{\phi(q)}C_\euB X
+O_{\xi,\euB,\eps}(X^{1/2+\eps})\bigg)\\
&=\frac{\chi(-1)\tau(\chi)}{q}\frac{\mu(q)}{\phi(q)}C_\euB X
\ssum{0<m<q\\(m,q)=1}\overline\chi(m)
+O_{q,\euB,\eps}(X^{1/2+\eps}).
}
Since $\chi$ is primitive and $q>1$, it follows that $\chi$
is non-principal, and thus
$$
\ssum{0<m<q\\(m,q)=1}\overline\chi(m)=0.
$$
Consequently, the first expression in the preceding estimate vanishes,
and we obtain \eqref{eq:goalbd} as required. 
This completes our proof that $(ii)\Longrightarrow(i)$.

\bigskip\subsection{Acknowledgment}
The author thanks Masatoshi Suzuki for pointing out the
relevant literature connecting Theorem~\ref{thm:RHvsGRH} to the
Linnik–Sprind\v zuk theorem.

\end{document}